\documentclass[a4paper,10pt]{article}

\usepackage{amsfonts}
\usepackage{amsthm}
\usepackage{amsmath}
\usepackage{amssymb}
\usepackage{mathtools} 
\usepackage{enumerate}

\usepackage[top=1in, bottom=1in, left=1in, right=1in]{geometry}
\usepackage{xcolor}

\usepackage{graphicx}
\usepackage{epstopdf}
\usepackage{subcaption}
\usepackage{authblk}

\theoremstyle{plain}
\newtheorem{theorem}{Theorem}[section]

\newtheorem{lemma}[theorem]{Lemma}

\theoremstyle{definition}
\newtheorem{definition}[theorem]{Definition}

\theoremstyle{remark}

\begin{document}


\title{A new representation for the solutions of fractional differential equations with variable coefficients}

\date{}

\author[1]{Arran Fernandez\thanks{Email: \texttt{arran.fernandez@emu.edu.tr}}}
\author[2,3]{Joel E. Restrepo\thanks{Email: \texttt{joel.restrepo@nu.edu.kz}}}
\author[2]{Durvudkhan Suragan\thanks{Email: \texttt{durvudkhan.suragan@nu.edu.kz}}}

\affil[1]{{\small Department of Mathematics, Faculty of Arts and Sciences, Eastern Mediterranean University, North Cyprus, via Mersin-10, Turkey}}
\affil[2]{{\small Department of Mathematics, Nazarbayev University, Astana, Kazakhstan}}
\affil[3]{{\small Department of Mathematics, University of Antioquia, Medellin, Colombia}}

\maketitle

\begin{abstract}
A recent development in the theory of fractional differential equations with variable coefficients has been a method for obtaining an exact solution in the form of an infinite series involving nested fractional integral operators. This solution representation is constructive but difficult to calculate in practice. Here we show a new representation of the solution function, as a convergent series of single fractional integrals, which will be easier to use for computational work and applications. In the particular case of constant coefficients, the solution is given in terms of the Mittag-Leffler function. We also show some applications in Cauchy problems for partial differential equations involving both time-fractional and space-fractional operators and with time-dependent coefficients. \\

\end{abstract}

\section{Introduction}

Fractional differential equations, describing relations between a function and its various derivatives of non-integer orders, have become a thriving field of research in the last few decades \cite{kilbasbook2006,podlubny}. Several analytical methods for solving classical ordinary or partial differential equations can be extended to the setting of fractional differential equations, such as Mikusi\'nski's operational calculus \cite{hadid-luchko,luchko,luchko-gorenflo}, transform methods for partial differential equations \cite{baleanu-fernandez,fernandez-baleanu-fokas}, weak solutions and spectral theory \cite{djida-area-nieto,djida-fernandez-area,cauchy1}, regularity estimates \cite{cauchy3,fernandez,otarola-salgado}, etc. One of the main reasons for studying fractional differential equations is their wide-ranging applications in physics, biology, engineering, economics, etc. \cite{hilfer,ionescu-etal,sun-etal}.

Just as with classical differential equations, there are different types of fractional differential equations with different levels of difficulty. Linear differential equations are almost always easier than nonlinear ones, and those with constant coefficients are easier than those where the coefficients are permitted to depend on one or more of the independent variables. Recently, a rigorous analytical method has been devised for solving fractional differential equations with variable coefficients. It has been applied, firstly to ordinary differential equations with fractional derivatives of Riemann--Liouville \cite{kim-o} and Caputo \cite{analitical} type, later to ordinary differential equations with other types of fractional operators \cite{FRS,RRS}, and also to partial differential equations with time-fractional and space-fractional terms \cite{BRS,joelsuragan}.

In the aforementioned papers, fractional ordinary differential equations with continuous variable coefficients were solved by means of an explicit representation for the solution, involving composition of Riemann--Liouville fractional integral operators, one inside another with multiplier functions. The resulting formulae are mathematically elegant but would be very difficult to calculate, even numerically, for a given problem with specific functions. Thus, the results in the literature so far are largely a mathematical curiosity, hard to apply in practice, even though the problems being solved (fractional differential equations with variable coefficients) do have real applications. In this paper, we construct a new formula for the same solution function, which will be better suited to numerical calculations and real-world applications.

Specifically, we treat the special case (of the general Caputo fractional differential equation analysed in \cite{analitical,RRS}) of the following fractional differential equation with one variable coefficient:
\begin{equation}\label{principal}
{^{C}D_{0+}^{\beta}}y(t)+a(t)y(t)=b(t),\qquad t\in[0,T],
\end{equation}
where $T>0$, $a,b\in C[0,T]$, $\beta\in\mathbb{C}$ with $\mathrm{Re}(\beta)>0$ (or simply $\beta>0$ real), and ${^{C}D_{0+}^{\beta}}$ is the Caputo fractional derivative (defined in a form that allows it to be applied on a suitably large function space, namely the expression \eqref{alternative} below). In the previous work of \cite{analitical}, the unique solution of \eqref{principal} with appropriate initial conditions was discovered to be
\[y(t)=\sum_{k=0}^{\infty}(-1)^kI^{\beta}_{0+}\Big(a(t)I_{0+}^{\beta}\Big)^kb(t).\]
Here we seek a new representation of this solution function which does not involve arbitrarily long compositions of Riemann--Liouville fractional integral operators. In fact, we will show an alternative formula for the solution which is given by an infinite series of single Riemann--Liouville fractional integrals of the forcing term $b(t)$, with the coefficient of this fractional integral depending just on the ordinary derivatives of the time-dependent coefficient $a(t)$. The new formula will be much easier to calculate and handle in practice, and we will demonstrate its usage in this paper, not only for fractional ordinary differential equations, but also for doubly fractional partial differential equations.

The paper is organised as follows. Section \ref{preli} is devoted to collecting definitions and results on fractional calculus and fractional differential equations. In Section \ref{mainresults}, we give the main results of the paper on the new representation of the solution of the fractional differential equation \eqref{principal}. In the special case of having a constant coefficient in \eqref{principal}, we show that the representation of the solution is given by the Mittag-Leffler function, which is consistent with already known results. In Section \ref{further}, using the results obtained in the previous section, we give an analytical solution of a Cauchy problem for a fractional partial differential equation with time-dependent coefficient. We finish the paper with some conclusions in Section \ref{Sec:concl}.

\section{Preliminaries}\label{preli}
 
In this section, we shall recall the definitions and some basic properties of the Riemann--Liouville fractional integro-differential operators -- for more details about these, see e.g. \cite[Chapter 2]{samko} and \cite[Chapter 2]{kilbasbook2006}. We shall also introduce the fractional differential equation to be analysed in this paper along with some rigorous facts about it. Finally, we shall introduce notations and basic facts concerning gamma functions and binomial coefficients, which will be used later in the paper.

\subsection{Riemann--Liouville fractional integro-differential operators}   

\begin{definition}[\cite{samko}, formula (2.17)]
Let $\alpha\in\mathbb{C}$, $\mathrm{Re}(\alpha)>0$, and let $f$ be an integrable function on a compact real interval $[a,b]$. The (left-sided) Riemann--Liouville fractional integral of $f$ is defined by:
\begin{equation}\label{fraci}
I_{a+}^{\alpha}f(t)=\frac1{\Gamma(\alpha)}\int_a^t (t-s)^{\alpha-1}f(s)\,\mathrm{d}s,\qquad t\in(a,b).
\end{equation}
\end{definition}

The Riemann--Liouville fractional integrals obey the following semigroup property \cite{kilbasbook2006,samko}:
\[I_{a+}^{\alpha}I_{a+}^{\beta}f(t)=I_{a+}^{\alpha+\beta}f(t),\qquad\mathrm{Re}(\alpha),\mathrm{Re}(\beta)>0,\,f\in L^1[a,b].
\]

\begin{definition}[\cite{samko}, formula (2.32)]
Let $\alpha\in\mathbb{C}$ with $\mathrm{Re}(\alpha)\geq0$, let $n=\lfloor\mathrm{Re}(\alpha)\rfloor+1$ so that $n-1\leq\mathrm{Re}(\alpha)<n$, and let $f\in AC^n[a,b]$ where $-\infty<a<b<\infty$. The (left-sided) Riemann--Liouville fractional derivative of $f$ is defined by:
\begin{equation}\label{fracd}
D_{a+}^{\alpha}f(t)=\left(\frac{\mathrm{d}}{\mathrm{d}t}\right)^n \big(I_{a+}^{n-\alpha}f\big)(t),\qquad t\in(a,b).
\end{equation}
\end{definition}

In this paper, instead of the Riemann--Liouville fractional derivative, we shall use the following modified version, for $\mathrm{Re}(\alpha)\geq0$ and $n=\lfloor\mathrm{Re}(\alpha)\rfloor+1$ and $f\in AC^n[a,b]$:
\begin{equation}\label{alternative}
^{C}D_{a+}^{\alpha}f(t)=D_{0+}^{\alpha}\biggl(f(t)-\sum_{j=0}^{n-1}\frac{f^{(j)}(a)}{j!}(t-a)^j\biggr),\qquad t\in(a,b).
\end{equation}
Notice that, if $f\in C^n[a,b]$, then the expression $^{C}D_{0+}^{\alpha}f(t)$ of \eqref{alternative} is exactly the so-called Caputo fractional derivative, defined as follows:  
\begin{equation}\label{Capder}
^{C}D^{\alpha}_{a+} {f(t)}=I_{a+}^{n-\alpha}\left(\frac{\mathrm{d}}{\mathrm{d}t}\right)^n f(t),\qquad t\in(a,b).
\end{equation}
The existence of the Caputo fractional derivative \eqref{Capder} is guaranteed by $f^{(n)}\in L^1[a,b]$, and the stronger condition $f\in C^n[a,b]$ gives the continuity of the derivative. Since the definitions \eqref{alternative} and \eqref{Capder} are identical for any function $f$ such that \eqref{Capder} is defined, it is reasonable to use \eqref{alternative} as the definition of the Caputo fractional derivative on the larger class of functions $AC^n[a,b]$, an extension of meaning of the Caputo fractional derivative beyond the space $C^n[a,b]$. Therefore, in the remainder of this paper we shall refer to the operator defined by \eqref{alternative} as the Caputo fractional derivative, without further elaboration.

Finally we give a fractional version of the Leibniz rule, various versions of which have been discussed in the literature. Results of this type have already proven useful in evaluating expressions involving nested fractional integrals with multiplier functions in between, e.g. in the 1972 paper of Love \cite{love}, although his work only covered the case where the multiplier is a power function. The following Lemma will be vital for our work below.

\begin{lemma}[\cite{miller-ross,osler,podlubny}] \label{Lem:FLR}
If $f$ is a continuous function and $g$ is an analytic function, then the following fractional Leibniz rule holds for any fractional integral of $f(t)g(t)$:
\[
I_{a+}^{\alpha}\big(f(t)g(t)\big)=\sum_{n=0}^{\infty}\binom{-\alpha}{n}I_{a+}^{\alpha+n}f(t)\cdot\left(\frac{\mathrm{d}}{\mathrm{d}t}\right)^ng(t),\qquad\mathrm{Re}(\alpha)>0.
\]
\end{lemma}

\subsection{Fractional differential equations with variable coefficients}

The following multi-term fractional differential equation with time-dependent continuous variable coefficients was investigated and solved in \cite{analitical}, with a more general version in \cite{RRS}: 
\begin{equation}\label{eq1}
^{C}D_{0+}^{\beta_0}y(t)+\sum_{i=1}^{m}a_i(t) ^{C}D_{0+}^{\beta_i}y(t)=b(t),\quad t\in[0,T],\; m\in\mathbb{N},
\end{equation}
under the initial conditions
\begin{equation}\label{eq2}
\left(\frac{\mathrm{d}}{\mathrm{d}t}\right)^k y(t)\Big|_{t=+0}=c_k\in\mathbb{R},\qquad k=0,1,\ldots,n_0-1, 
\end{equation}
where $a_i,b\in C[0,T]$, $\mathrm{Re}(\beta_i)\geq0$, and $n_i=\lfloor\mathrm{Re}(\beta_i)\rfloor+1$ for $i=0,1,\ldots,m-1$, and $\mathrm{Re}(\beta_0)>\mathrm{Re}(\beta_1)>\ldots>\mathrm{Re}(\beta_m)\geqslant0$ where if $\mathrm{Re}(\beta_m)=0$ then we assume $\mathrm{Im}(\beta_m)=0$ as well. As a special case of the problem above, equation \eqref{eq1} was also studied under homogeneous initial conditions:
\begin{equation}\label{eq4}
\left(\frac{\mathrm{d}}{\mathrm{d}t}\right)^k y(t)\Big|_{t=+0}=0,\qquad k=0,1,\ldots,n_0-1.
\end{equation}
The solution found in \cite{RRS} was given in the following special function space: \[C^{n_0-1,\beta_0}[0,T]:=\biggl\{y\in C^{n_0-1}[0,T]\; : \; ^{C}D_{0+}^{\beta_0}y\in C[0,T]\biggr\},\]
endowed with the norm
\[\big\|y\big\|_{C^{n_0-1,\beta_0}[0,T]}=\sum_{k=0}^{n_0-1}\left\|\left(\frac{\mathrm{d}}{\mathrm{d}t}\right)^k y\right\|_{C[0,T]}+\Big\|\,^{C}D_{0+}^{\beta_0}y\Big\|_{C[0,T]}.\]
According to \cite[Theorem 3.8]{RRS} and \cite[Remark 2.7]{RRS}, we have the following result.

\begin{theorem}\label{thm3.1}
The initial value problem given by \eqref{eq1} and \eqref{eq4}, with all notation and terminology as defined above, has a unique solution $y\in C^{n_0-1,\beta_0}[0,T]$, and it is given by the following formula:
\begin{equation}\label{for27}
y(t)=\sum_{k=0}^{+\infty}(-1)^k I_{0+}^{\beta_0}\left(\sum_{i=1}^{m}a_i(t)I_{0+}^{\beta_0-\beta_i}\right)^k b(t).
\end{equation}
\end{theorem}

Notice that the representation \eqref{for27} involves the composition of Riemann--Liouville fractional integrals with function multipliers. In the case of constant coefficients, due to linearity and composition properties of Riemann--Liouville fractional integrals, we can get a closed-form representation in terms of the Mittag-Leffler function \cite[Theorem 4.3]{RRS}. For the general case of time-dependent continuous variable coefficients, it seems that representation \eqref{for27} cannot be improved, since the variable coefficients appear inside the composition of fractional integrals.

In this article, we shall consider the case of one variable coefficient, namely the following equation with $\mathrm{Re}(\beta)>0$ and $a,b\in C[0,T]$:
\begin{equation}\label{eq1here}
^{C}D_{0+}^{\beta}y(t)+a(t)y(t)=b(t),\qquad t\in[0,T],
\end{equation}
under the initial conditions
\begin{equation}\label{eq4here}
\left(\frac{\mathrm{d}}{\mathrm{d}t}\right)^k y(t)\Big|_{t=+0}=0,\qquad k=0,1,\ldots,n-1,\quad n=\lfloor\mathrm{Re}(\beta)\rfloor+1.
\end{equation}
From Theorem \ref{thm3.1} above, it follows that the fractional differential equation \eqref{eq1here} and \eqref{eq4here} has a unique solution $y\in C^{n-1,\beta}[0,T]$, where $n-1\leq\mathrm{Re}(\beta)<n$, and it is represented by:
\[y(t)=\sum_{k=0}^{\infty}(-1)^kI^{\beta}_{0+}\Big(a(t)I_{0+}^{\beta}\Big)^kb(t)=I^{\beta}_{0+}\left(\sum_{k=0}^{\infty}(-1)^k\Big(a(t)I_{0+}^{\beta}\Big)^kb(t)\right).\]
Our task in this paper is to show that this representation of the solution can be improved and given in a more explicit form, avoiding the composition of Riemann--Liouville fractional integrals given by $(a(t)I_{0+}^{\beta})^k$, which will be more suitable for explicit calculation and approximation.

\subsection{Gamma functions and related topics}

\begin{definition}[\cite{whittaker-watson}]
The gamma function $\Gamma(z)$ is defined by
\[
\Gamma(z)=\int_0^{\infty}t^{z-1}e^{-t}\,\mathrm{d}t,\qquad\mathrm{Re}(z)>0,
\]
and by analytic continuation using the functional equation $\Gamma(z+1)=z\Gamma(z)$ for $\mathrm{Re}(z)\leq0$. This defines a function which is analytic on the whole complex plane except the points $z=0,-1,-2,-3,\dots$, while the function $\frac1{\Gamma(z)}$ is an entire function of $z\in\mathbb{C}$, with zeros at the singularities of the gamma function.
\end{definition}

Throughout this paper, we will use the notation $\binom{\alpha}{\beta}$ and $\alpha!$ to denote the following expressions:
\[
\binom{\alpha}{\beta}=\frac{\alpha!}{\beta!(\alpha-\beta)!}=\frac{\Gamma(\alpha+1)}{\Gamma(\beta+1)\Gamma(\alpha-\beta+1)}\quad\quad\text{ and }\quad\quad\alpha!=\Gamma(\alpha+1),
\]
even when the numbers $\alpha$ and $\beta$ are not necessarily natural numbers. We will also frequently make use of the fact that
\[
\sum_{n=0}^{\infty}\frac{\cdots}{(n-k)!(N-n)!\cdots}=\sum_{n=k}^{N}\frac{\cdots}{n!(N-n)!\cdots},
\]
where the $\cdots$ represent any possible quantity (constant or variable, depending on anything we want, but the same on both sides of the equation). This identity, reducing an infinite series to a finite one, is valid since dividing by a factorial of a negative integer always gives zero, due to the zeros of the function $\frac1{\Gamma(z)}$.

We shall also frequently make use of the following lemma on series.

\begin{lemma} \label{Lem:hypergeom}
Let $a,b,c\in\mathbb{C}$ be constants. The series
\[
\sum_{n=0}^{\infty}\frac{1}{(a-n)!(b-n)!(c+n)!n!}
\]
converges to the value
\[
\frac{(a+b+c)!}{a!b!(a+c)!(b+c)!},
\]
provided that either $\mathrm{Re}(a+b+c)>-1$ or one of $a,b$ is in $\mathbb{N}$.
\end{lemma}

\begin{proof}
Consider the standard power series for the hypergeometric function $\prescript{}{2}F_1$, which has radius of convergence $1$:
\[
\prescript{}{2}F_1(-a,-b;1+c;z)=\sum_{n=0}^{\infty}\frac{\Gamma(-a+n)\Gamma(-b+n)\Gamma(1+c)}{\Gamma(-a)\Gamma(-b)\Gamma(1+c+n)}\cdot\frac{z^n}{n!}.
\]
Gauss's hypergeometric theorem states that the value of this function at $z=1$, provided that $\mathrm{Re}(1+c)>\mathrm{Re}(-a-b)$, is given by
\[
\prescript{}{2}F_1(-a,-b;1+c;1)=\frac{\Gamma(1+c)\Gamma(1+c+a+b)}{\Gamma(1+c+a)\Gamma(1+c+b)}.
\]

A standard identity on the gamma function is
\[
\frac{\Gamma(x+n)}{\Gamma(x)}=(-1)^n\cdot\frac{\Gamma(1-x)}{\Gamma(1-x-n)},\quad x\in\mathbb{C},\;n\in\mathbb{Z}.
\]
It is easy to prove from writing each quotient of gamma functions as a finite product of terms: the left-hand side is $x(x+1)(x+2)\dots(x+n-1)$ and the right-hand side similarly. This also justifies the validity of the identity for all $x\in\mathbb{C}$, without needing to avoid singularities of the gamma function.

Therefore, we have
\begin{align*}
\sum_{n=0}^{\infty}\frac{1}{(a-n)!(b-n)!(c+n)!n!}&=\sum_{n=0}^{\infty}\frac{1}{\Gamma(1+a-n)\Gamma(1+b-n)\Gamma(1+c+n)n!} \\
&=\sum_{n=0}^{\infty}\frac{\Gamma(-a+n)\Gamma(-b+n)}{\Gamma(1+a)\Gamma(-a)\Gamma(1+b)\Gamma(-b)\Gamma(1+c+n)n!} \\
&=\frac{\prescript{}{2}F_1(-a,-b;1+c;1)}{\Gamma(1+a)\Gamma(1+b)\Gamma(1+c)} \\
&=\frac{1}{\Gamma(1+a)\Gamma(1+b)\Gamma(1+c)}\cdot\frac{\Gamma(1+c)\Gamma(1+c+a+b)}{\Gamma(1+c+a)\Gamma(1+c+b)} \\
&=\frac{(a+b+c)!}{a!b!(a+c)!(b+c)!}.
\end{align*}
So the infinite series is convergent and the result follows in the case that $\mathrm{Re}(a+b+c)>-1$.

If we assume that one of $a,b$ is in $\mathbb{N}$, then the series is actually finite, since all but finitely many of the terms are zero. Fixing (without loss of generality) $a=N\in\mathbb{N}$, the result becomes
\begin{equation}
\label{N:analcont}
\sum_{n=0}^{N}\frac{1}{(N-n)!(b-n)!(c+n)!n!}=\frac{(N+b+c)!}{N!b!(N+c)!(b+c)!},
\end{equation}
which we now know is true under the assumption $\mathrm{Re}(N+b+c)>-1$. The variables $b$ and $c$ are still free in the complex plane subject to this restriction. Both sides of the equation \eqref{N:analcont} are analytic in $b$ and $c$, since the series is finite and therefore convergent, so we can use analytic continuation to deduce that \eqref{N:analcont} is valid for all $b,c\in\mathbb{C}$.
\end{proof}

\section{Main results}\label{mainresults}

The starting point is the following formula:
\begin{equation}
\label{yformula}
y(t)=\sum_{k=0}^{\infty}(-1)^kI^{\beta}_{0+}\Big(a(t)I_{0+}^{\beta}\Big)^kb(t),
\end{equation}
where $a(t)$, $b(t)$ are continuous functions on $[0,T]$ and $\mathrm{Re}(\beta)>0$. Here, in order to be able to use the fractional Leibniz rule of Lemma \ref{Lem:FLR}, we must also assume that $a(t)$ is an analytic function. (The textbook of Miller \& Ross \cite[p. 97]{miller-ross} mentions unpublished work of E. R. Love in which he weakened the analyticity assumption for the fractional Leibniz rule, but we have been unable to find this work, if indeed it was ever published.)

Let us use the following notation for the $k$-summand:
\[
S_k(t):=I^{\beta}_{0+}\Big(a(t)I_{0+}^{\beta}\Big)^kb(t).
\]
The goal of this section is to find a formula for the summand $S_k(t)$, and hence for the solution function $y(t)$, as an infinite series without involving any composition of fractional integrals.

\subsection{Early cases}

Let us start by considering the first few cases of $k$, in order to spot a pattern and write down a formula for general $k$ to be proved by induction on $k$.

\medskip

\textbf{The case $\boldsymbol{k=0}$.} Here the result is trivial: \[S_0(t)=I_{0+}^{\beta}b(t).\]

\medskip

\textbf{The case $\boldsymbol{k=1}$.} Here the result follows from a single application of the fractional Leibniz rule:
\begin{align*}
S_1(t)&=I^{\beta}_{0+}\Big(a(t)I_{0+}^{\beta}b(t)\Big)=\sum_{n=0}^{\infty}\binom{-\beta}{n}\Big[a^{(n)}(t)\Big]\Big[I_{0+}^{2\beta+n}b(t)\Big]
\end{align*}

\medskip

\textbf{The case $\boldsymbol{k=2}$.} Here we need to use the fractional Leibniz rule and then the classical Leibniz rule:
\begin{align*}
S_2(t)&=I^{\beta}_{0+}\Big(a(t)S_1(t)\Big)=I^{\beta}_{0+}\left[\sum_{n=0}^{\infty}\binom{-\beta}{n}\Big[a(t)a^{(n)}(t)\Big]\Big[I_{0+}^{2\beta+n}b(t)\Big]\right] \\
&=\sum_{n=0}^{\infty}\sum_{m=0}^{\infty}\binom{-\beta}{n}\binom{-\beta}{m}D^m\Big(a(t)D^na(t)\Big)\cdot I_{0+}^{3\beta+n+m}b(t) \\
&=\sum_{n=0}^{\infty}\sum_{m=0}^{\infty}\sum_{i=0}^{m}\binom{-\beta}{n}\binom{-\beta}{m}\binom{m}{i}D^{i}a(t)\cdot D^{m+n-i}a(t)\cdot I_{0+}^{3\beta+n+m}b(t).
\end{align*}
Putting $p=n+m$, and rearranging the sums so that the $i$-sum is outside the $m$-sum, we have:
\begin{align*}
S_2(t)&=\sum_{p=0}^{\infty}\sum_{i=0}^{\infty}\sum_{m=i}^{p}\binom{-\beta}{p-m}\binom{-\beta}{m}\binom{m}{i}D^{i}a(t)\cdot D^{p-i}a(t)\cdot I_{0+}^{3\beta+p}b(t).
\end{align*}
The innermost sum (over $m$) can be simplified using Lemma \ref{Lem:hypergeom}:
\begin{align*}
\sum_{m=i}^{p}\binom{-\beta}{p-m}\binom{-\beta}{m}\binom{m}{i}&=\sum_{m=i}^{p}\frac{(-\beta)!(-\beta)!}{(-\beta-p+m)!(p-m)!(-\beta-m)!(m-i)!i!} \\
&=\sum_{n=0}^{p-i}\frac{(-\beta)!(-\beta)!}{(-\beta-p+n+i)!(p-n-i)!(-\beta-n-i)!n!i!} \\
&=\frac{(-\beta)!(-\beta)!}{i!}\sum_{n=0}^{\infty}\frac{1}{(p-i-n)!(-\beta-i-n)!(-\beta-p+i+n)!n!} \\
&=\frac{(-\beta)!(-\beta)!}{i!}\cdot\frac{(-2\beta-i)!}{(p-i)!(-\beta-i)!(-\beta)!(-2\beta-p)!}.
\end{align*}
Therefore, re-labelling $p$ as $n$, we have the final answer for $k=2$:
\begin{align*}
S_2(t)&=\sum_{n=0}^{\infty}\left[\sum_{i=0}^{n}\frac{(-2\beta-i)!(-\beta)!}{(-2\beta-n)!(-\beta-i)!(n-i)!i!}D^{i}a(t)\cdot D^{n-i}a(t)\right] I_{0+}^{3\beta+n}b(t).
\end{align*}

\medskip

\textbf{The case $\boldsymbol{k=3}$.} Here the manipulation is even more complicated. We start by applying the fractional Leibniz rule and then the classical Leibniz rule for a product of three functions:
\begin{align*}
S_3(t)&=I^{\beta}_{0+}\Big(a(t)S_2(t)\Big) \\
&=I^{\beta}_{0+}\left[\sum_{n=0}^{\infty}\sum_{i=0}^{n}\frac{(-2\beta-i)!(-\beta)!}{(-2\beta-n)!(-\beta-i)!(n-i)!i!}a(t)\cdot D^{i}a(t)\cdot D^{n-i}a(t)\cdot I_{0+}^{3\beta+n}b(t)\right] \\
&=\sum_{n=0}^{\infty}\sum_{m=0}^{\infty}\sum_{i=0}^{\infty}\binom{-\beta}{m}\frac{(-2\beta-i)!(-\beta)!}{(-2\beta-n)!(-\beta-i)!(n-i)!i!}D^m\Big[a(t)\cdot D^{i}a(t)\cdot D^{n-i}a(t)\Big] I_{0+}^{4\beta+n+m}b(t) \\
&=\sum_{n=0}^{\infty}\sum_{m=0}^{\infty}\sum_{i=0}^{\infty}\frac{(-2\beta-i)!(-\beta)!(-\beta)!}{(-2\beta-n)!(-\beta-m)!(-\beta-i)!(n-i)!i!m!}D^m\Big[a(t)\cdot D^{i}a(t)\cdot D^{n-i}a(t)\Big] I_{0+}^{4\beta+n+m}b(t) \\
&=\sum_{n=0}^{\infty}\sum_{m=0}^{\infty}\sum_{i=0}^{\infty}\sum_{r_1+r_2\leq m}\frac{(-2\beta-i)!(-\beta)!(-\beta)!}{(-2\beta-n)!(-\beta-m)!(-\beta-i)!(n-i)!i!m!}\cdot\frac{m!}{r_1!r_2!(m-r_1-r_2)!} \\ &\hspace{5cm}\times D^{r_1}a(t)\cdot D^{i+r_2}a(t)\cdot D^{n-i+m-r_1-r_2}a(t)\cdot I_{0+}^{4\beta+n+m}b(t).
\end{align*}
Putting $p=m+n$, and rearranging sums so that the $m$-series is the innermost one, we get:
\begin{align*}
&S_3(t)=\sum_{p=0}^{\infty}\sum_{r_1=0}^{\infty}\sum_{r_2=0}^{\infty}\sum_{i=0}^{\infty}\sum_{m=r_1+r_2}^{p}\frac{(-2\beta-i)!(-\beta)!(-\beta)!}{(-2\beta-p+m)!(-\beta-m)!(-\beta-i)!(p-m-i)!i!r_1!r_2!(m-r_1-r_2)!} \\ &\hspace{5cm}\times D^{r_1}a(t)\cdot D^{i+r_2}a(t)\cdot D^{p-i-r_1-r_2}a(t)\cdot I_{0+}^{4\beta+p}b(t) \\
&=\sum_{p=0}^{\infty}\sum_{r_1=0}^{\infty}\sum_{r_2=0}^{\infty}\sum_{i=0}^{\infty}\left[\sum_{m'=0}^{p-r_1-r_2}\frac{1}{(-2\beta-p+m'+r_1+r_2)!(-\beta-m'-r_1-r_2)!(p-m'-r_1-r_2-i)!m'!}\right] \\ &\hspace{3cm}\times\frac{(-2\beta-i)!(-\beta)!(-\beta)!}{(-\beta-i)!i!r_1!r_2!} D^{r_1}a(t)\cdot D^{i+r_2}a(t)\cdot D^{p-i-r_1-r_2}a(t)\cdot I_{0+}^{4\beta+p}b(t).
\end{align*}
By Lemma \ref{Lem:hypergeom}, the inner sum over $m'$ is
\[
\frac{(-3\beta-r_1-r_2-i)!}{(-\beta-r_1-r_2)!(p-r_1-r_2-i)!(-3\beta-p)!(-2\beta-i)!},
\]
so we have
\begin{multline*}
S_3(t)=\sum_{p=0}^{\infty}\sum_{r_1=0}^{\infty}\sum_{r_2=0}^{\infty}\sum_{i=0}^{\infty}\frac{(-3\beta-r_1-r_2-i)!(-\beta)!(-\beta)!}{(-\beta-r_1-r_2)!(p-r_1-r_2-i)!(-3\beta-p)!(-\beta-i)!i!r_1!r_2!} \\ \times D^{r_1}a(t)\cdot D^{i+r_2}a(t)\cdot D^{p-i-r_1-r_2}a(t)\cdot I_{0+}^{4\beta+p}b(t).
\end{multline*}
Putting $i_1=r_1$ and $i_2=i+r_2$, this becomes
\begin{align*}
S_3(t)&=\sum_{p=0}^{\infty}\sum_{i_1=0}^{\infty}\sum_{i_2=0}^{\infty}\sum_{i=0}^{\infty}\frac{(-3\beta-i_1-i_2)!(-\beta)!(-\beta)!}{(-\beta-i_1-i_2+i)!(p-i_1-i_2)!(-3\beta-p)!(-\beta-i)!i!i_1!(i_2-i)!} \\ &\hspace{5cm}\times D^{i_1}a(t)\cdot D^{i_2}a(t)\cdot D^{p-i_1-i_2}a(t)\cdot I_{0+}^{4\beta+p}b(t) \\
&=\sum_{p=0}^{\infty}\sum_{i_1=0}^{\infty}\sum_{i_2=0}^{\infty}\left[\sum_{i=0}^{\infty}\frac{1}{(-\beta-i_1-i_2+i)!(-\beta-i)!i!(i_2-i)!}\right] \\ &\hspace{2cm}\times\frac{(-3\beta-i_1-i_2)!(-\beta)!(-\beta)!}{(p-i_1-i_2)!(-3\beta-p)!i_1!} D^{i_1}a(t)\cdot D^{i_2}a(t)\cdot D^{p-i_1-i_2}a(t)\cdot I_{0+}^{4\beta+p}b(t).
\end{align*}
By Lemma \ref{Lem:hypergeom} again, the inner sum over $i$ is
\[
\frac{(-2\beta-i_1)!}{(-\beta)!i_2!(-2\beta-i_1-i_2)!(-\beta-i_1)!}.
\]
Therefore, re-labelling $p$ as $n$, we have the final answer for $k=3$:
\begin{align*}
S_3(t)&=\sum_{p=0}^{\infty}\sum_{i_1=0}^{\infty}\sum_{i_2=0}^{\infty}\frac{(-3\beta-i_1-i_2)!(-2\beta-i_1)!(-\beta)!}{(-3\beta-p)!(-2\beta-i_1-i_2)!(-\beta-i_1)!(p-i_1-i_2)!i_1!i_2!} \\ &\hspace{5cm}\times D^{i_1}a(t)\cdot D^{i_2}a(t)\cdot D^{p-i_1-i_2}a(t)\cdot I_{0+}^{4\beta+p}b(t) \\
&=\sum_{n=0}^{\infty}\Bigg[\sum_{i_1+i_2\leq n}\frac{(-3\beta-i_1-i_2)!(-2\beta-i_1)!(-\beta)!}{(-3\beta-n)!(-2\beta-i_1-i_2)!(-\beta-i_1)!(n-i_1-i_2)!i_1!i_2!} \\ &\hspace{5cm}\times D^{i_1}a(t)\cdot D^{i_2}a(t)\cdot D^{n-i_1-i_2}a(t)\Bigg] I_{0+}^{4\beta+n}b(t).
\end{align*}

\subsection{The general case}

Having solved the problem for the first few values of $k$, we are now in a position to extrapolate the above results and guess a formula for general $k$, which we can then prove by induction.

\begin{theorem}
\label{Thm:Sk}
For all $k\geq1$, with all notation defined as above, we have
\begin{multline}
\label{newrep}
S_k(t)=\sum_{n=0}^{\infty}\Bigg[\sum_{i_1+i_2+\dots+i_k=n}\frac{(-\beta-n)!(-2\beta-i_1)!(-3\beta-i_1-i_2)!\dots(-k\beta-i_1-\dots-i_{k-1})!}{(-\beta-i_1)!(-2\beta-i_1-i_2)!\dots(-(k-1)\beta-i_1-\dots-i_{k-1})!(-k\beta-n)!} \\\times\frac{n!}{i_1!\dots i_k!}a^{(i_1)}(t)\dots a^{(i_k)}(t)\Bigg]\binom{-\beta}{n}I_{0+}^{(k+1)\beta+n}b(t).
\end{multline}
\end{theorem}

\begin{proof}
We proceed by induction on $k$. The $k=1$ case of \eqref{newrep} is
\[
S_1(t)=\sum_{n=0}^{\infty}\Bigg[\sum_{i=n}\frac{(-\beta-n)!}{(-\beta-i)!} \cdot\frac{n!}{i!}a^{(i)}(t)\Bigg]\binom{-\beta}{n}I_{0+}^{2\beta+n}b(t)=\sum_{n=0}^{\infty}a^{(n)}(t)\binom{-\beta}{n}I_{0+}^{2\beta+n}b(t),
\]
and the $k=2$ case of \eqref{newrep} is
\begin{align*}
S_2(t)&=\sum_{n=0}^{\infty}\Bigg[\sum_{i_1+i_2=n}\frac{(-\beta-n)!(-2\beta-i_1)!}{(-\beta-i_1)!(-2\beta-n)!} \cdot\frac{n!}{i_1!i_2!}a^{(i_1)}(t)a^{(i_2)}(t)\Bigg]\binom{-\beta}{n}I_{0+}^{3\beta+n}b(t) \\
&=\sum_{n=0}^{\infty}\Bigg[\sum_{i=0}^{n}\frac{(-\beta-n)!(-2\beta-i)!}{(-\beta-i)!(-2\beta-n)!} \binom{n}{i}a^{(i)}(t)a^{(n-i)}(t)\Bigg]\binom{-\beta}{n}I_{0+}^{3\beta+n}b(t),
\end{align*}
both of which are correct according to the work done above.

Now let us assume the result is true for $S_k(t)$, and prove it for $S_{k+1}(t)$. Firstly, using the fractional Leibniz rule and then the classical Leibniz rule:
\begin{align*}
S_{k+1}(t)&=I^{\beta}_{0+}\Big(a(t)S_k(t)\Big) \\
&=I^{\beta}_{0+}\Bigg[\sum_{n=0}^{\infty}\binom{-\beta}{n}I_{0+}^{(k+1)\beta+n}b(t)\sum_{i_1+i_2+\dots+i_k=n}\frac{n!}{i_1!\dots i_k!}a(t)a^{(i_1)}(t)\dots a^{(i_k)}(t) \\&\hspace{2cm}\times\frac{(-\beta-n)!(-2\beta-i_1)!(-3\beta-i_1-i_2)!\dots(-k\beta-i_1-\dots-i_{k-1})!}{(-\beta-i_1)!(-2\beta-i_1-i_2)!\dots(-(k-1)\beta-i_1-\dots-i_{k-1})!(-k\beta-n)!}\Bigg] \\
&=\sum_{n=0}^{\infty}\sum_{m=0}^{\infty}\binom{-\beta}{n}\binom{-\beta}{m}I_{0+}^{(k+2)\beta+n+m}b(t)\sum_{i_1+i_2+\dots+i_k=n}\frac{n!}{i_1!\dots i_k!}D^m\Big[a(t)a^{(i_1)}(t)\dots a^{(i_k)}(t)\Big] \\&\hspace{2cm}\times\frac{(-\beta-n)!(-2\beta-i_1)!(-3\beta-i_1-i_2)!\dots(-k\beta-i_1-\dots-i_{k-1})!}{(-\beta-i_1)!(-2\beta-i_1-i_2)!\dots(-(k-1)\beta-i_1-\dots-i_{k-1})!(-k\beta-n)!} \\
&=\sum_{n=0}^{\infty}\sum_{m=0}^{\infty}\binom{-\beta}{n}\binom{-\beta}{m}I_{0+}^{(k+2)\beta+n+m}b(t)\sum_{i_1+i_2+\dots+i_k=n}\frac{n!}{i_1!\dots i_k!} \\ &\hspace{2cm}\times\sum_{r_1+r_2+\dots+r_{k+1}=m}\frac{m!}{r_1!\dots r_{k+1}!}a^{(r_1)}(t)a^{(i_1+r_2)}(t)\dots a^{(i_k+r_{k+1})}(t) \\&\hspace{2cm}\times\frac{(-\beta-n)!(-2\beta-i_1)!(-3\beta-i_1-i_2)!\dots(-k\beta-i_1-\dots-i_{k-1})!}{(-\beta-i_1)!(-2\beta-i_1-i_2)!\dots(-(k-1)\beta-i_1-\dots-i_{k-1})!(-k\beta-n)!} \\
&=\sum_{n=0}^{\infty}\sum_{m=0}^{\infty}I_{0+}^{(k+2)\beta+n+m}b(t)\sum_{\substack{i_1+i_2+\dots+i_k=n \\ r_1+r_2+\dots+r_{k+1}=m}}\frac{1}{i_1!\dots i_k!r_1!\dots r_{k+1}!} a^{(r_1)}(t)a^{(i_1+r_2)}(t)\dots a^{(i_k+r_{k+1})}(t) \\&\hspace{1cm}\times\frac{(-\beta)!(-\beta)!(-2\beta-i_1)!(-3\beta-i_1-i_2)!\dots(-k\beta-i_1-\dots-i_{k-1})!}{(-\beta-i_1)!(-2\beta-i_1-i_2)!\dots(-(k-1)\beta-i_1-\dots-i_{k-1})!(-k\beta-n)!(-\beta-m)!}
\end{align*}
Now, putting $p=n+m$ and then separating out $i_k$ and $r_{k+1}$ among the others, we have:
\begin{align*}
S_{k+1}(t)&=\sum_{p=0}^{\infty}\sum_{n=0}^{\infty}I_{0+}^{(k+2)\beta+p}b(t)\sum_{\substack{i_1+i_2+\dots+i_k=n \\ r_1+r_2+\dots+r_{k+1}=p-n}}\frac{1}{i_1!\dots i_k!r_1!\dots r_{k+1}!} a^{(r_1)}(t)a^{(i_1+r_2)}(t)\dots a^{(i_k+r_{k+1})}(t) \\&\hspace{1cm}\times\frac{(-\beta)!(-\beta)!(-2\beta-i_1)!(-3\beta-i_1-i_2)!\dots(-k\beta-i_1-\dots-i_{k-1})!}{(-\beta-i_1)!(-2\beta-i_1-i_2)!\dots(-(k-1)\beta-i_1-\dots-i_{k-1})!(-k\beta-n)!(-\beta-p+n)!} \\
&=\sum_{p=0}^{\infty}I_{0+}^{(k+2)\beta+p}b(t)\sum_{n=0}^{\infty}\sum_{i_1,\dots,i_{k-1}}\sum_{r_1,\dots,r_k} a^{(r_1)}(t)a^{(i_1+r_2)}(t)\dots a^{(i_{k-1}+r_k)}(t)a^{(p-I-R)}(t) \\ &\hspace{1cm}\times\frac{1}{i_1!\dots i_{k-1}!(n-I)!r_1!\dots r_k!(p-n-R)!} \\&\hspace{1cm}\times\frac{(-\beta)!(-\beta)!(-2\beta-i_1)!(-3\beta-i_1-i_2)!\dots(-k\beta-i_1-\dots-i_{k-1})!}{(-\beta-i_1)!(-2\beta-i_1-i_2)!\dots(-(k-1)\beta-i_1-\dots-i_{k-1})!(-k\beta-n)!(-\beta-p+n)!},
\end{align*}
where we have written $I=i_1+\dots+i_{k-1}$ and $R=r_1+\dots+r_k$ for simplicity. Now we can make the $n$-summation the innermost one and use Lemma \ref{Lem:hypergeom} to say that
\begin{align*}
&{\color{white}=}\sum_{n=0}^{\infty}\frac{1}{(n-I)!(p-n-R)!(-k\beta-n)!(-\beta-p+n)!} \\
&=\sum_{n'=0}^{\infty}\frac{1}{n'!(p-n'-I-R)!(-k\beta-n'-I)!(-\beta-p+n'+I)!} \\
&=\frac{(-(k+1)\beta-I-R)!}{(p-I-R)!(-k\beta-I)!(-\beta-R)!(-(k+1)\beta-p)!}.
\end{align*}
Re-labelling $p$ as $n$, and comparing the expression we have reached so far with the desired final result, we see that it remains to prove the following for all $n\geq0$:
\begin{multline*}
\sum_{i_1,\dots,i_{k-1}}\sum_{r_1,\dots,r_k} \frac{(-\beta)!(-\beta)!(-2\beta-i_1)!(-3\beta-i_1-i_2)!\dots(-(k-1)\beta-i_1-\dots-i_{k-2})!}{(-\beta-i_1)!(-2\beta-i_1-i_2)!\dots(-(k-1)\beta-i_1-\dots-i_{k-1})!} \\ \times\frac{(-(k+1)\beta-I-R)!}{(-\beta-R)!(-(k+1)\beta-n)!} \\ \times\frac{1}{i_1!\dots i_{k-1}!r_1!\dots r_k!(n-I-R)!} a^{(r_1)}(t)a^{(i_1+r_2)}(t)\dots a^{(i_{k-1}+r_k)}(t)a^{(n-I-R)}(t)\\
=\sum_{j_1+j_2+\dots+j_{k+1}=n}\frac{(-\beta)!(-2\beta-j_1)!(-3\beta-j_1-j_2)!\dots(-(k+1)\beta-j_1-\dots-j_k)!}{(-\beta-j_1)!(-2\beta-j_1-j_2)!\dots(-k\beta-j_1-\dots-j_k)!(-(k+1)\beta-n)!} \\\times\frac{1}{j_1!\dots j_{k+1}!}a^{(j_1)}(t)\dots a^{(j_{k+1})}(t).
\end{multline*}
Writing $j_1=r_1$, $j_2=i_1+r_2$, \dots, $j_k=i_{k-1}+r_k$, this becomes
\begin{multline*}
\sum_{j_1,\dots,j_k}\sum_{i_1\leq j_2,\dots,i_{k-1}\leq j_k}\frac{(-\beta)!(-\beta)!(-2\beta-i_1)!(-3\beta-i_1-i_2)!\dots(-(k-1)\beta-i_1-\dots-i_{k-2})!}{(-\beta-i_1)!(-2\beta-i_1-i_2)!\dots(-(k-1)\beta-i_1-\dots-i_{k-1})!} \\ \times\frac{(-(k+1)\beta-j_1-\dots-j_k)!}{(-\beta-j_1-\dots-j_k+i_1+\dots+i_{k-1})!(-(k+1)\beta-n)!} \\ \times\frac{1}{i_1!\dots i_{k-1}!j_1!(j_2-i_1)!\dots (j_k-i_{k-1})!(n-j_1-\dots-j_k)!} a^{(j_1)}(t)\dots a^{(j_k)}(t)a^{(n-j_1-\dots-j_k)}(t)\\
=\sum_{j_1,\dots,j_k}\frac{(-\beta)!(-2\beta-j_1)!(-3\beta-j_1-j_2)!\dots(-(k+1)\beta-j_1-\dots-j_k)!}{(-\beta-j_1)!(-2\beta-j_1-j_2)!\dots(-k\beta-j_1-\dots-j_k)!(-(k+1)\beta-n)!} \\\times\frac{1}{j_1!\dots j_k!(n-j_1-\dots-j_k)!}a^{(j_1)}(t)\dots a^{(n-j_1-\dots-j_k)}(t).
\end{multline*}
Therefore, it remains to prove the following for all $j_1,\dots,j_k\geq0$:
\begin{multline}
\label{toprove:j}
\sum_{i_1,\dots,i_{k-1}}\frac{(-\beta)!(-2\beta-i_1)!(-3\beta-i_1-i_2)!\dots(-(k-1)\beta-i_1-\dots-i_{k-2})!}{(-\beta-i_1)!(-2\beta-i_1-i_2)!\dots(-(k-1)\beta-i_1-\dots-i_{k-1})!} \\ \times\frac{1}{(-\beta-j_1-\dots-j_k+i_1+\dots+i_{k-1})!i_1!\dots i_{k-1}!(j_2-i_1)!\dots (j_k-i_{k-1})!} \\
=\frac{(-2\beta-j_1)!(-3\beta-j_1-j_2)!\dots(-k\beta-j_1-\dots-j_{k-1})!}{(-\beta-j_1)!(-2\beta-j_1-j_2)!\dots(-k\beta-j_1-\dots-j_k)!j_2!\dots j_k!},
\end{multline}
which is an identity between finite series, elementary albeit complicated to prove.

In the series on the left-hand side of \eqref{toprove:j}, there are four denominator terms and no numerator terms involving $i_{k-1}$. By Lemma \ref{Lem:hypergeom}, we have
\begin{multline*}
\sum_{i_{k-1}=0}^{\infty}\frac{1}{(-(k-1)\beta-i_1-\dots-i_{k-1})!(-\beta-j_1-\dots-j_k+i_1+\dots+i_{k-1})!i_{k-1}!(j_k-i_{k-1})!} \\ =\frac{(-k\beta-j_1-\dots-j_{k-1})!}{(-(k-1)\beta-i_1-\dots-i_{k-2})!j_k!(-k\beta-j_1-\dots-j_k)!(-\beta-j_1-\dots-j_{k-1}+i_1+\dots+i_{k-2})!},
\end{multline*}
so \eqref{toprove:j} is equivalent to
\begin{multline*}
\sum_{i_1,\dots,i_{k-1}}\frac{(-\beta)!(-2\beta-i_1)!(-3\beta-i_1-i_2)!\dots(-(k-2)\beta-i_1-\dots-i_{k-3})!}{(-\beta-i_1)!(-2\beta-i_1-i_2)!\dots(-(k-2)\beta-i_1-\dots-i_{k-2})!} \\ \times\frac{1}{(-\beta-j_1-\dots-j_{k-1}+i_1+\dots+i_{k-2})!i_1!\dots i_{k-2}!(j_2-i_1)!\dots (j_{k-1}-i_{k-2})!} \\
=\frac{(-2\beta-j_1)!(-3\beta-j_1-j_2)!\dots(-(k-1)\beta-j_1-\dots-j_{k-2})!}{(-\beta-j_1)!(-2\beta-j_1-j_2)!\dots(-(k-1)\beta-j_1-\dots-j_{k-1})!j_2!\dots j_{k-1}!},
\end{multline*}
which is itself identical to \eqref{toprove:j} with $k$ replaced by $k-1$. Therefore, by finite descent, it will be sufficient to prove \eqref{toprove:j} in the basic case $k=2$. In this case, the equation \eqref{toprove:j} is
\[
\sum_{i}\frac{(-\beta)!}{(-\beta-i)!}\cdot\frac{1}{(-\beta-j_1-j_2+i)!i!(j_2-i)!}=\frac{(-2\beta-j_1)!}{(-\beta-j_1)!(-2\beta-j_1-j_2)!j_2!},
\]
which follows from Lemma \ref{Lem:hypergeom}. Now we have completed the induction process, and the proof is complete.
\end{proof}

We now establish our main result of the paper, which is a direct consequence of Theorem \ref{Thm:Sk}.

\begin{theorem}\label{thmnewR}
Let $a(t)$ be analytic and $b(t)$ be continuous on $[0,T]$, and let $\beta\in\mathbb{C}$ with $\mathrm{Re}(\beta)>0$ and $m=\lfloor\mathrm{Re}(\beta)\rfloor+1$. Then the fractional differential equation \eqref{eq1here} under the initial conditions \eqref{eq4here} has a unique solution $y\in C^{m-1,\beta}[0,T]$ and it is given by the following convergent infinite series:
\begin{multline*}
y(t)=\sum_{k=0}^{\infty}\sum_{n=0}^{\infty}(-1)^k\binom{-\beta}{n}I_{0+}^{(k+1)\beta+n}b(t) \\ \times\Bigg[\sum_{i_1+i_2+\dots+i_k=n}\frac{(-\beta-n)!(-2\beta-i_1)!(-3\beta-i_1-i_2)!\dots(-k\beta-i_1-\dots-i_{k-1})!}{(-\beta-i_1)!(-2\beta-i_1-i_2)!\dots(-(k-1)\beta-i_1-\dots-i_{k-1})!(-k\beta-n)!}\Bigg. \\
\times \Bigg.\frac{n!}{i_1!\dots i_k!}a^{(i_1)}(t)\dots a^{(i_k)}(t)\Bigg].
\end{multline*}
\end{theorem}

\subsection{The constant-coefficient case} \label{Subsec:CC}

In order to verify the consistency of our new results with some existing results in the literature, let us take the special case of \eqref{eq1here} where the variable coefficient is a constant, i.e. $a(t)=\lambda\in\mathbb{R}$. Thus, we consider the initial value problem
\begin{equation}\label{eqhereC}
\begin{cases}
^{C}D_{0+}^{\beta}y(t)+\lambda y(t)=b(t),\qquad t\in[0,T], \\
\left(\frac{\mathrm{d}}{\mathrm{d}t}\right)^k y(t)\Big|_{t=+0}=0,\qquad k=0,1,\ldots,\lfloor\beta\rfloor,
\end{cases}
\end{equation}
where $\mathrm{Re}(\beta)>0$, $\lambda\in\mathbb{R}$, and $b\in C[0,T]$.

The consequence of Theorem \ref{thmnewR} in this case is that the fractional differential equation \eqref{eqhereC} has a unique solution $y\in C^{\lfloor\beta\rfloor,\beta}[0,T]$ given by the following formula:
\begin{equation}
\label{CCsoln}
y(t)=\int_0^t (t-s)^{\beta-1}E_{\beta,\beta}(-\lambda(t-s)^{\beta})b(s)\,\mathrm{d}s,
\end{equation}
where $E_{\alpha,\beta}(z)=\sum_{k=0}^{\infty}\frac{z^k}{\Gamma(\alpha k+\beta)}$ is the two-parameter Mittag-Leffler function \cite{mittagbook} for any complex parameters $\alpha,\beta$ with $\mathrm{Re}\beta>0$.

To see why, notice that, setting $a(t)=\lambda$ in Theorem \ref{thmnewR}, we have directly
\[
y(t)=\sum_{k=0}^{\infty}(-\lambda)^kI_{0+}^{(k+1)\beta}b(t),
\]
since the second series given in Theorem \ref{thmnewR} becomes zero for $n\neq 0$, while for $n=0$ we get just $\lambda^k$. Now the desired representation follows straightforwardly using the definition of the Mittag-Leffler function.

The expression \eqref{CCsoln} for the solution to the constant-coefficient initial value problem \eqref{eqhereC}, in terms of the Mittag-Leffler function, is consistent with the results of \cite[Theorem 7.2 and Remark 7.1]{diethelm} for the solution of the same differential equation, found using the variation of constants method. Therefore, we have verified the consistency of our results with those already known in the literature.

\section{Applications to partial differential equations}\label{further}

The results of the previous section concern ordinary differential equations of fractional type, but they can also be applied to solve certain types of partial differential equations using both time-fractional and space-fractional operators. In this section we give an analytical solution of a Cauchy type problem for a fractional partial differential equation with time-dependent coefficient. Normally, the solutions of such problems are found by applying numerical tools or different methods to approximate the solution. For instance, the general problems treated in \cite{cauchy3,cauchy1,cauchy2} and references therein can be compared with our Cauchy type problems treated below. Finding explicitly the solution of a fractional partial differential equation with a time-dependent variable coefficient, as we shall do in this section, will be potentially advantageous in the understanding of these problems.

First, we recall the Fourier transform, the inverse Fourier transform, and the fractional Laplacian, which will be used below.

\begin{definition}
The Fourier transform $f$ of a function $\phi:\mathbb{R}^n\to\mathbb{C}$ is defined by
\[f(y)=(\mathcal{F\phi})(y)=\widehat{\phi}(y)=\int_{\mathbb{R}^n}e^{iy\cdot x}\phi(x)\,\mathrm{d}x.\]
Conversely, the inverse Fourier transform is defined by 
\[\phi(y)=\big(\mathcal{F}^{-1}f\big)(y)=\frac{1}{(2\pi)^n}\int_{\mathbb{R}^n}e^{-iy\cdot\tau}f(\tau)\,\mathrm{d}\tau,\] 
where $``\cdot"$ is the inner product of vectors in $\mathbb{R}^n$.  
\end{definition}

\begin{definition}
The fractional Laplacian $(-\Delta)^{\lambda}$ is defined \cite[Chapter 5]{samko} as a pseudo-differential operator with the symbol $|y|^{2\lambda}$, namely by:
\begin{equation}\label{fractionallaplacian}
(\mathcal{F}(-\Delta)^{\lambda}f)(y)=|y|^{2\lambda}(\mathcal{F}f)(y),\qquad y\in\mathbb{R}^n.
\end{equation}
It can be also defined by the following hypersingular integral: 
\[(-\Delta)^{\lambda}f(y)=\frac{1}{d_{n,m}(\lambda)}\int_{\mathbb{R}^n}\frac{(\Delta_u^m f)(y)}{|u|^{n+2\lambda}}du,\]
whenever $0<\lambda<m$, $m\in\mathbb{N}$, where $(\Delta_u^m f)(y)$ is the difference operator given in \cite[Formulas 25.57 and 25.58]{samko}, and $d_{n,m}(\lambda)$ is a normalisation constant. Notice that for $\lambda=1$ we recover the classical Laplacian in $\mathbb{R}^n$, i.e. $\Delta_x=\sum_{k=1}^{n}\partial_{x_k}^{2}$.
\end{definition}


\begin{theorem} \label{Thm:PDE}
Let $T>0$, $0<\alpha\leq 1$, $\beta>0$, $m-1\leq\beta<m$, $r(x,\cdot)\in C[0,T]$, and let $\Psi(t)\in C^{\infty}[0,T]$ be analytic. We consider the following Cauchy problem for a fractional partial differential equation: 
\begin{equation}\label{eq1Cauchynew}
\begin{cases}
^{C}\partial_{t}^{\beta}h(x,t)+\Psi(t)(-\Delta)^{\alpha}h(x,t)&=r(x,t),\quad t\in (0,T],\;x\in\mathbb{R}^n, \\
\hfill h(x,t)\big|_{t=0}&=0, \\
\hfill\partial_t h(x,t)\big|_{t=0}&=0, \\
&\hspace{0.2cm}\vdots \\
\hfill\partial_t^{m} h(x,t)\big|_{t=0}&=0,
\end{cases}
\end{equation}
where $(-\Delta)^{\alpha}$ is the fractional Laplacian with respect to $x\in\mathbb{R}^n$ and $^{C}\partial_{t}^{\beta}$ is the Caputo fractional derivative (defined according to \eqref{alternative} with $a=0$) with respect to $t\in(0,T)$.

The solution of the fractional Cauchy type problem \eqref{eq1Cauchynew} is given explicitly by
\begin{multline*}
h(x,t)=\sum_{k=0}^{\infty}\sum_{n=0}^{\infty}\binom{-\beta}{n}I_{t}^{(k+1)\beta+n}\mathcal{F}_s^{-1}\big((-|s|^{2\alpha})^k\widehat{r}(s,t)\big)(x,t) \\ \times\Bigg[\sum_{i_1+i_2+\dots+i_k=n}\frac{(-\beta-n)!(-2\beta-i_1)!(-3\beta-i_1-i_2)!\dots(-k\beta-i_1-\dots-i_{k-1})!}{(-\beta-i_1)!(-2\beta-i_1-i_2)!\dots(-(k-1)\beta-i_1-\dots-i_{k-1})!(-k\beta-n)!}\Bigg. \\
\times \Bigg.\frac{n!}{i_1!\dots i_k!}\Psi^{(i_1)}(t)\dots\Psi^{(i_k)}(t)\Bigg].
\end{multline*}
\end{theorem}

\begin{proof}
We begin by applying the space Fourier transform to the original problem \eqref{eq1Cauchynew}, which transforms it into the following Cauchy problem in the $(s,t)$ domain:
\begin{equation}\label{fourierpronew}
\begin{cases}
^{C}\partial_{t}^{\beta}\widehat{h}(s,t)+|s|^{2\alpha}\Psi(t)\widehat{h}(s,t)&=\widehat{r}(s,t),\quad t\in(0,T],\; s\in\mathbb{R}^n,\\
\hfill\widehat{h}(s,t)\big|_{t=0}&=0, \\
\hfill\partial_t \widehat{h}(s,t)\big|_{t=0}&=0,\\
&\hspace{0.2cm}\vdots \\
\hfill\partial_t^{m} \widehat{h}(s,t)\big|_{t=0}&=0.
\end{cases}
\end{equation}
Here in \eqref{fourierpronew}, we have essentially a fractional ordinary differential equation with variable coefficient $|s|^{2\alpha}\Psi(t)$. Hence, by using Theorem \ref{thmnewR} with $a(t)=|s|^{2\alpha}\Psi(t)$, we have that the system \eqref{fourierpronew} has a unique solution in the space $C^{\beta,m-1}[0,T]$ given by: 
\begin{multline*}
\widehat{h}(s,t)=\sum_{k=0}^{\infty}\sum_{n=0}^{\infty}(-|s|^{2\alpha})^k\binom{-\beta}{n}I_{0+}^{(k+1)\beta+n}\widehat{r}(s,t) \\ \times\Bigg[\sum_{i_1+i_2+\dots+i_k=n}\frac{(-\beta-n)!(-2\beta-i_1)!(-3\beta-i_1-i_2)!\dots(-k\beta-i_1-\dots-i_{k-1})!}{(-\beta-i_1)!(-2\beta-i_1-i_2)!\dots(-(k-1)\beta-i_1-\dots-i_{k-1})!(-k\beta-n)!}\Bigg. \\
\times \Bigg.\frac{n!}{i_1!\dots i_k!}\Psi^{(i_1)}(t)\dots\Psi^{(i_k)}(t)\Bigg].
\end{multline*}
We then obtain the desired representation of the solution by applying the inverse Fourier transform to the above formula for $\widehat{h}(s,t)$.
\end{proof}

As a corollary of Theorem \ref{Thm:PDE}, considering the case $\Psi(t)=\lambda\in\mathbb{R}$ in the equation \eqref{eq1Cauchynew}, we get the following fractional Cauchy problem with a constant coefficient:
\begin{equation}\label{eq1CauchynewC}
\begin{cases}
^{C}\partial_{t}^{\beta}h(x,t)+\lambda(-\Delta)^{\alpha} h(x,t)&=r(x,t),\quad t\in (0,T]$,\;$x\in\mathbb{R}^n, \\
\hfill h(x,t)\big|_{t=0}&=0, \\
\hfill\partial_t h(x,t)\big|_{t=0}&=0, \\
&\hspace{0.2cm}\vdots \\
\hfill\partial_t^{m} h(x,t)\big|_{t=0}&=0,
\end{cases}
\end{equation}
where $0<\alpha\leq 1$, $\beta>0$, $m-1<\beta\leq m,$ $r(x,\cdot)\in C[0,T]$, and $\Psi(t)\in C[0,T]$. The solution is given by 
\[h(x,t)=\int_0^t (t-s)^{\beta-1}\mathcal{F}_s^{-1}\biggl\{E_{\beta,\beta}(-\lambda|s|^{2\alpha}(t-s)^{\beta})\widehat{r}(s,t)\biggr\}\,\mathrm{d}s,\]
where the Mittag-Leffler function emerges in the same way as in Section \ref{Subsec:CC} above, and where it is assumed that $r(\cdot,t)\in L^{1}(\mathbb{R}^n)$ and $E_{\beta,\beta}(-\lambda|\cdot|^{2\alpha}(t-s)^{\beta})\widehat{r}(\cdot,t)\in L^{1}(\mathbb{R}^n)$.

\section{Conclusions and future work} \label{Sec:concl}

In this paper, we have established new series representations for the solutions to fractional differential equations with variable time-dependent coefficients. These exact solutions had already been constructed in the previous work of \cite{analitical}, and their uniqueness in appropriate function spaces had been verified in \cite{RRS}, but the solutions constructed in the literature so far have been of a form which is very difficult to compute in practice, involving arbitrarily many nested fractional integrals with function multipliers in between. Our work here showcases a new formula which involves a single fractional integral in each summand, with a coefficient which depends only on the classical derivatives of the variable coefficient function. This is expected to be very useful in numerical calculations of solution functions for such equations.

To illustrate the impact of our obtained results, we have applied them, not only for ordinary differential equations with Caputo fractional derivatives and variable coefficients, but also for partial differential equations with both time-fractional and space-fractional derivative operators and with time-dependent coefficients. We have also verified that, in the case of constant coefficients, our results are consistent with those already known in the literature \cite{diethelm} for this simple case.

The work in this paper has been under the assumption of homogeneous initial conditions. If we want to consider initial conditions different from zero, we can easily extend the results here, in a similar way as it was done in \cite{analitical,RRS}, to obtain a new version of Theorem \ref{thmnewR} giving the solution of the problem \eqref{eq1here} under general initial conditions, as well as a new version of Theorem \ref{Thm:PDE} giving the solution of the problem \eqref{eq1Cauchynew} under general initial conditions. We leave these trivial variations of our results to be proved by an interested reader.

Many other extensions of these results are also possible. The differential equations considered here have featured only one fractional time derivative and one time-dependent coefficient. A more general class of fractional differential equations was solved in the original work of \cite{analitical}, and we hope in the future to find a new series representation for that more general solution too, analogous to our work here. Furthermore, some other recent papers \cite{FRS,RRS} have focused on extensions of the work of \cite{analitical} to differential equations with other types of fractional operators, and these too can now be studied using the methods we have displayed in this paper. The results presented here have opened up many avenues for future research.

\section{Acknowledgements} The authors were supported by the Nazarbayev University Program 091019CRP2120. Joel E. Restrepo thanks to Colciencias and Universidad de Antioquia (Convocatoria 848 - Programa de estancias postdoctorales 2019) for their support.

\end{document}